\documentclass[11pt, eqno]{article}
\usepackage{bbm}
\usepackage{mathrsfs}
\usepackage{amsfonts}
\usepackage{amssymb}
\usepackage{graphicx}
\usepackage[all]{xy}
\usepackage{amsthm}
\usepackage{amsmath}
\usepackage{amsmath,amssymb,latexsym,color}
\usepackage[mathscr]{eucal}
\usepackage{CJK}
\usepackage{cases}
\usepackage{graphics}

\usepackage{psfrag}
\usepackage{subfigure}
\usepackage{color}

\usepackage{amssymb,latexsym}
\usepackage{amsmath,latexsym}
\usepackage{amscd}

\newtheorem{theorem}{Theorem}[section]

\newtheorem{corollary}[theorem]{Corollary}

\newtheorem{definition}[theorem]{Definition}
\newtheorem{example}[theorem]{Example}
\newtheorem{remark}[theorem]{Remark}

\newtheorem{lemma}[theorem]{Lemma}
\newtheorem{question}[theorem]{Question}

\begin{document}


\title{The Poisson bracket invariant for open covers consisting of topological disks on surfaces}
\date{June 6, 2022}
\author{Kun Shi and Guangcun Lu
\thanks{Corresponding author
\endgraf \hspace{2mm} 2020 {\it Mathematics Subject Classification}:
 53D50, 57K20.
\endgraf \hspace{2mm}  {\it Key words and phrases}:
 Poisson bracket, Poisson bracket invariant, topological disk, closed surface.
\endgraf\hspace{2mm} 
Supported by the NNSF  11271044 of China
 and the Fundamental Research Funds for the Central Universities grant.
}}
 \maketitle \vspace{-0.3in}


\begin{abstract}
L. Buhovsky, A. Logunov and S. Tanny proved the (strong) Poisson bracket conjecture by Leonid Polterovich in dimension $2$.
In this note, instead of open cover consisting of displaceable sets in their work,
 we consider open cover constituted of topological discs
and give a necessary and sufficient condition that Poisson bracket invariants of these covers are positive.
\end{abstract}
\vspace{-0.1in}
\medskip\vspace{12mm}

\maketitle \vspace{-0.5in}

\section{Introduction and results}
\setcounter{equation}{0}

 Throughout this note, we always assume that $(M,\omega)$ is a closed connected symplectic manifold
 and that $\|\cdot\|:C^{\infty}(M)\rightarrow [0,\infty)$ denotes the $C^0$-norm or $L^{\infty}$-norm of smooth functions on $M$.
  Recall that a subset
$X$ of $M$ is called {\it displaceable} if there exists a Hamiltonian diffeomorphism
$\phi\in{\rm Ham}(M,\omega)$ such that $\phi(X)$ is disjoint with $X$ (\cite{Ho90}).
Under the assumption that $(M,\omega)$ is also semi-positive,
 Entov and Polterovich \cite{EnPo06} proved a surprising link between
  non-displaceability and Poisson commutativity,
the so-called non-displaceable fiber theorem, which claimed that if a smooth map $\vec{f}=(f_1,\cdots,f_N):M\to\mathbb{R}^N$ has
   pairwise Poisson commuting  coordinate functions $f_i$, $i=1,\cdots, N$, then there exists one fiber $(\vec{f})^{-1}(p)$ which is non-empty and
 non-displaceable.
   In the book \cite{PoR14} by Polterovich and Rosen, the assumption of semipositivity
   was removed out (see \cite[Theorem~6.1.8]{PoR14}), and there exists the following
 rigidity of partitions of unity (see \cite{EnPoZa07}): any finite open cover of $M$  by open displaceable
sets does not admit a Poisson commuting partition of unity.

We say an open cover $\mathcal{U}$ of $M$ to be
 {\it connected} (resp. {\it displaceable}, resp.
 {\it connected and displaceable})
if each element of $\mathcal{U}$ is connected (resp.
displaceable, resp. connected and displaceable) in $M$.

 In order to measure the Poisson noncommutativity of a smooth partition of unity
 $\mathcal{F}:=\{f_i\}_{i=1}^N$  subordinated to a  connected and displaceable
finite open cover $\mathcal{U}:=\{U_i\}_{i=1}^N$ of $M$, Polterovich \cite{Po12} defined the {\it magnitude} of its Poisson non-commutativity by
 $$
\nu_c(\mathcal{F}):=\max_{a,b\in[-1,1]^N}\left\|\left\{\sum_{i=1}^{N}a_if_i,\sum_{j=1}^{N}b_jf_j\right\}\right\|,
$$
 and the {\it Poisson bracket invariant} of $\mathcal{U}$ by
$$
{\rm pb}(\mathcal{U}):=\inf_{\mathcal{F}}\nu_c(\mathcal{F}),
$$
where the infimum is taken over all partitions of unity $\mathcal{F}$ subordinate to
$\mathcal{U}$.  Polterovich \cite{Po14} gave several lower bounds for ${\rm pb}(\mathcal{U})$ and proposed:

\begin{question}[\hbox{\rm \cite[Question~8.1]{Po14}}]\label{q:1}
{\rm Is it true that ${\rm pb}(\mathcal{U})\ge C/e(\mathcal{U})$, where the constant $C$ depends only on
the symplectic manifold $(M, \omega)$, and where $e(\mathcal{U}):=\max_{i\in I}e(U_i)$ and
$e(U_i)$ is the displacement energy of $U_i$?}
\end{question}

This was called the {\it strong Poisson bracket conjecture}
in \cite{Pay19}, and studied in references \cite{BT17, BLT17, BEP12, Ga18, Is15, LaPa19,  Pay18, Pay19, Se15}.

Recently, Buhovsky, Logunov and Tanny affirmatively answered this question in dimension $2$ (see \cite{BLT17}).
Let $\mathcal{U}$ be a finite open cover on smooth manifold $M$.
A set $U$ in $\mathcal{U}$ is called {\it essential}  if $\mathcal{U}\setminus\{U\}$ is not a cover (see \cite[Def.1.6]{BLT17}).
 Denote by $\mathcal{J}(\mathcal{U})\subset\mathcal{U}$ the collection
  of essential sets of $\mathcal{U}$. Let
$|\mathcal{U}|$ denotes the number of open subsets in $\mathcal{U}$.
A subcover $\mathcal{U}_0$ of ~$\mathcal{U}$~ is said to be {\it smallest} if $|\mathcal{U}_0|$ is equal to
$$
\kappa(\mathcal{U}):=\min_{\mathcal{U}'\subset\mathcal{U}}\{|\mathcal{U}'|\;\big|\;\mathcal{U}'\; \text{is also a cover of }\; M\}.
$$
 Clearly, the smallest subcover  of ~$\mathcal{U}$~ is not necessarily unique, and each of them is
 a minimal cover of $M$. Moreover, a subcover  of $\mathcal{U}$ which is  a minimal cover of $M$ is not always a smallest subcover
 of $\mathcal{U}$, since the cardinality of a minimal cover of $M$ may be greater than $\kappa(\mathcal{U})$.
 Each element of $\mathcal{U}_0$ is essential for the cover $\mathcal{U}_0$, but it not always  essential for $\mathcal{U}$.
  Buhovsky, Logunov and Tanny proved the following results.

\begin{theorem}[\hbox{\cite[Theorem~1.7 and Remark~1.8]{BLT17}}]\label{th:1.2}
Let $(M,\omega)$ be a closed and connected symplectic surface. Let $\mathcal{U}=\{U_i\}_{i\in I}$ be a finite open cover of $M$ by topological discs of area less than ${\rm Area}(M)/2$, and let $\{f_i\}_{i\in I}$ be any partition of unity subordinate to $\mathcal{U}$. Then there exists an absolute constant $c>0$, which is independent of $(M,\omega)$, $\mathcal{U}$ and $\{f_i\}_{i\in I}$,
 such that
$$
{\rm pb}(\mathcal{U})\ge \frac{|\mathcal{J}(\mathcal{U})|c}{{\rm Area}(M)}\quad\hbox{and}\quad
{\rm pb}(\mathcal{U})\ge \frac{c}{\min_{U\in\mathcal{J}(\mathcal{U})}{\rm Area}(U)},
$$
where we set the minimum to be infinity if $\mathcal{J}(\mathcal{U})$ is empty.
\end{theorem}

\begin{theorem}[\hbox{\cite[Theorem~$1.5^\prime$]{BLT17}}]\label{th:1.2.1}
	Let $ (M,\omega) $ be a closed and connected symplectic surface. Let $\mathcal U=\{U_i\}_{i\in I}$, $\mathcal V=\{V_j\}_{j\in J}$ be {finite}
	open covers of $M$, and let $\{f_i\}_{i\in I}$,  $\{g_j\}_{j\in J}$ be partitions of unity subordinate to $\mathcal U$, $\mathcal V$ correspondingly. Then,
	\begin{equation}\label{eq:gen_cover-reformulation}
	\int_M \sum_{i\in I}\sum_{j\in J} |\{f_i,g_j\}|\ \omega \geq \frac{{\rm Area}(M)}{2\cdot \max(e(\mathcal U), e(\mathcal V))}.
	\end{equation}
\end{theorem}

Note that the definition of displaceability at the bottom of \cite[page~1]{BLT17}
is slightly different from ours. They called a subset
$X$ of $M$ {\it displaceable} if the closure $\overline{X}$ of $X$ is displaceable in our definition above.
Thus, in our language, \cite[Corollary~1.9]{BLT17} may be reformulated as:

\begin{corollary}[\hbox{\cite[Corollary~1.9]{BLT17}}]\label{cor:1.2.2}
	Let $ (M,\omega) $ be a closed and connected symplectic surface, and let $\mathcal{U}=\{U_i\}_{i\in I}$ be
a finite connected open cover of $M$ such that the closure of each $U_i$ is  displaceable. Then for an absolute constant $c>0$, it holds that
	\begin{equation}\label{eq:pb-bound}
	{\rm pb}(\mathcal{U})\geq \frac{c}{e(\mathcal U)}.
	\end{equation}
\end{corollary}

In Theorem~\ref{th:1.2}, since  the second inequality  holds trivially if
 $\mathcal{J}(\mathcal{U})=\emptyset$,  it is necessary  to assume that
the cover $\mathcal{U}$  contains at least one essential set
so that $\min_{U\in\mathcal{J}(\mathcal{U})}{\rm Area}(U)$ is a positive number.
Note also that the condition
\begin{center}
``$\max_{1\le i\le N}{\rm Area}(U_i)<{\rm Area}(M)/2$"
\end{center}
in Theorem~\ref{th:1.2} implies $\kappa(\mathcal{U})\ge 3$ (because $M$ is closed and each $U_i$ is
a topological disc of area less than ${\rm Area}(M)/2$). It is obvious that the converse is not true in general.

\begin{example}\label{ex:converse}
{\rm Let $\omega$ be the standard area form on
$$
M=S^2=\{(x_1,x_2,x_3)\in\mathbb{R}^3\,|\,x_1^2+x_2^2+x_3^2=1\},
$$
and $\mathcal{U}=\{U_1,U_2,U_3\}$, where
 \begin{equation}\label{e:converse}
\left.\begin{array}{ll}
&U_1=\{(x_1,x_2,x_3)\in S^2\,|\,x_3<1/2\},\\
& U_2=\{(x_1,x_2,x_3)\in S^2\,|\,x_3>0,\;x_1>-1/4\},\\
&U_3=\{(x_1,x_2,x_3)\in S^2\,|\,x_3>0,\;x_1<1/4\}.
\end{array}\right\}
\end{equation}
Clearly,  $\kappa(\mathcal{U})=3$, but ${\rm Area}(U_1)>{\rm Area}(M)/2$.}
\end{example}

%

The assumptions in Theorem~\ref{th:1.2} can be weakened. In fact, we have
the following result.

\begin{theorem}\label{th:1.3}
Let $ (M,\omega) $ be a closed and connected symplectic surface, and let $\mathcal{U}=\{U_i\}_{i\in I}$ be
a finite open cover of $M$ made by topological discs and with $\kappa(\mathcal{U})\ge 3$. Then
there exists an absolute constant $c>0$ such that
$$
{\rm pb}(\mathcal{U})\ge \frac{|\mathcal{J}(\mathcal{U})|c}{{\rm Area}(M)}\quad\hbox{and}\quad
{\rm pb}(\mathcal{U})\ge \frac{c}{\min_{U\in\mathcal{J}(\mathcal{U})}{\rm Area}(U)}.
$$
\end{theorem}


As an improvement of a special case of  \cite[Theorem~1.5 or $1.5^\prime$]{BLT17}, we have
the following analogue of Theorem~\ref{th:1.2.1}
(\hbox{\cite[Theorem~$1.5^\prime$]{BLT17}}).

\begin{theorem}\label{th:1.4}
Let $(M,\omega)$ be a closed connected symplectic surface and let
$\mathcal{U}=\{U_i\}_{i=1}^{N}$ be a finite open cover consisting of topological discs. If $\kappa(\mathcal{U})\ge 3$, then for any partition of unity subordinate to $\mathcal{U}$, $\mathcal{F}:=\{f_i\}_{i=1}^{N}$, there holds
$$
\int_M\sum_{i,j=1}^N|\{f_i,f_j\}|\omega\ge 2.
$$
Moreover, there exists an absolute constant $C>0$ such that
\begin{equation}\label{e:lower-bound}
{\rm pb}(\mathcal{U})\ge \frac{C}{{\rm Area}(M)}.
\end{equation}
\end{theorem}

\begin{remark}
{\rm Since the Lusternik-Schnirelmann category of every closed surface of genus $g \ge 1$ is at least $3$,
the condition $\kappa(\mathcal{U})\ge 3$ in Theorem~\ref{th:1.4} is automatically satisfied in such closed surfaces.
Hence (\ref{e:lower-bound}) holds. In this case (\ref{e:lower-bound}) may also follow from
\cite[Theorem~1.4.14]{Pay19} as pointed out by Payette \cite{Pay19+}. J. Payette called (\ref{e:lower-bound})
 the {\it weak Poisson bracket conjecture}
}(see \cite{Pay19}).
\end{remark}

Clearly, if an open cover $\mathcal{U}=\{U_i\}_{i=1}^{N}$  of  a closed connected symplectic surface
 $(M,\omega)$  consists of topological discs, then $\kappa(\mathcal{U})\ge 2$. We claim that ${\rm pb}(\mathcal{U})=0$ if  $\kappa(\mathcal{U})<3$.
In fact, in this case, any smallest subcover $\mathcal{U}_0$  of $\mathcal{U}$ only contains two sets, saying
$\mathcal{U}_0=\{U_1,U_2\}$, and any smooth partition of unity  $\mathcal{F}_0=\{f_1,f_2\}$ subordinated to $\mathcal{U}_0$ satisfies $\{f_1,f_2\}=\{f_1,1-f_1\}=0$ and so
  $0={\rm pb}(\mathcal{U}_0)\ge {\rm pb}(\mathcal{U})\ge 0$.
From these and Theorem~\ref{th:1.4}, we arrive at:

\begin{corollary}\label{cor:1.10}
Let $(M,\omega)$ be a closed connected symplectic surface and $\mathcal{U}=\{U_i\}_{i=1}^{N}$ be a open cover of $M$ consisting of topological discs.
Then ${\rm pb}(\mathcal{U})>0$ if and only if $\kappa(\mathcal{U})\ge 3$.
\end{corollary}

Under the assumptions of Corollary~\ref{cor:1.2.2} (\hbox{\cite[Corollary~1.9]{BLT17}}), the inequality in (\ref{e:lower-bound}) may follow from (\ref{eq:pb-bound}), since $e(\mathcal{U})\le{\rm Area}(M,\omega)/2$.
Observe that the assumptions of Theorem~\ref{th:1.4} do not imply the assumptions of Corollary~\ref{cor:1.2.2}, and vice-versa.
Actually, from Theorem~\ref{th:1.4} we can also derive:

\begin{corollary}\label{cor:1.11}
	Let $(M,\omega) $ be a closed and connected symplectic surface and let
$\mathcal{U}=\{U_i\}^N_{i=1}$ be a finite connected open
	 cover of $M$ such that the closure of each $U_i$ is
 displaceable. Then for an absolute constant $c>0$, it holds that
	\begin{equation}\label{eq:pb-bound1}
	{\rm pb}(\mathcal{U})\geq \frac{c}{{\rm Area}(M,\omega)}.
	\end{equation}
\end{corollary}
\begin{proof}
Since $U_i$ is connected, so is the closure $\overline{U}_i$ of it.
Moreover, $\overline{U}_i$ is  displaceable, and compact due to the fact that $M$ is a closed surface.
By \cite[Remark~1.2]{BLT17},
    we have embedded open topological discs  $V_i\subset M $ with smooth boundaries
   such that  $U_i\subset V_i$ and ${\rm Area}(V_i) \le {\rm Area}(M)/2$, $i=1,\cdots,N$.
  In fact, by Exercise 3 in \cite[Chapter 3, \S1]{Hi94}, $\overline{U}_i$ has a closed neighborhood
  $W_i$ which is a smooth submanifold of $M$.
  Moreover, since $W_i$ can be required to be arbitrarily close $\overline{U}_i$,
 we may assume that it is displaceable, connected and has area less than ${\rm Area}(M)/2$.
 As Payette \cite{Pay19+} told us,  using the Jordan-Schoenflies theorem and  Exercise 3 in \cite[Chapter 3, \S1]{Hi94}
   may yield a  smoothly embedded closed disc $\overline{D}_i\supseteq W_i$ of area ${\rm Area}(\overline{D}_i)\le {\rm Area}(M)/2$
(see also \cite[Theorem~4.1(2)-(i)]{Pay18} for a detailed proof).
   Then $V_i:={\rm Int}(\overline{D}_i)$ satisfies our requirements.
  Clearly, $\mathcal{V}=\{V_i\}^N_{i=1}$ is an open cover of $M$, and $\kappa(\mathcal{V})\ge 3$. Hence
 (\ref{eq:pb-bound1}) follows from Theorem~\ref{th:1.4}, since
  ${\rm pb}(\mathcal{U})\ge {\rm pb}(\mathcal{V})$ by definition.
 \end{proof}

Finally, it should be pointed out that Example~\ref{ex:converse} satisfies the conditions of Theorem~\ref{th:1.4},
but does not those of Corollary~\ref{cor:1.2.2} (\hbox{\cite[Corollary~1.9]{BLT17}}).
Notice that  the cover $\mathcal{U}=\{U_1, U_2, U_3\}$ in Example~\ref{ex:converse}
is $3$-localized (\cite{Pay18, Pay19} for the precise definition) because
we can  choose three points $p_i\in S^2$, $i=1,2,3$ so that each $p_i$ only belongs to $U_i$ for $i=1,2,3$.
 For example, set
 $$
p_1=(0,0,-1),\quad p_2=(\sqrt{7}/4, 0, 3/4),\quad p_3=(-\sqrt{7}/4, 0, 3/4).
$$
Hence, the weak Poisson bracket conjecture for the cover $\mathcal{U}$ in Example~\ref{ex:converse}
can also follows from Theorem 1.4.12 of \cite{Pay19}. However, we can add an open topological disk
 $U_4\subset S^2$ to $\mathcal{U}$ so that the new cover $\mathcal{V}:=\{U_1,U_2,U_3,U_4\}$ of $S^2$
 also satisfies $\kappa(\mathcal{V})=3$, but cannot be $3$-localized.
Hence, the weak Poisson bracket conjecture for the cover $\mathcal{V}$ cannot follows from Theorem 1.4.12 of \cite{Pay19}.
Indeed,  by (\ref{e:converse}) it is direct to check that
\begin{equation}\label{e:converse+}
\left.\begin{array}{ll}
&U_1\cap U_2^c\cap U_3^c=\{(x_1,x_2,x_3)\in S^2\,|\,x_3\le 0\},\\
& U_2\cap U_1^c\cap U_3^c=\{(x_1,x_2,x_3)\in S^2\,|\,x_3\ge 1/2,\;x_1\ge 1/4\},\\
&U_3\cap U_1^c\cap U_2^c=\{(x_1,x_2,x_3)\in S^2\,|\,x_3\ge 1/2,\;x_1\le-1/4\}
\end{array}\right\}
\end{equation}
and hence that the complementary set of union of these three subsets in $S^2$ is
$$
\left\{(x_1,x_2,x_3)\in S^2\,\Big|\,0<x_3<\frac{1}{2},\;|x_1|\le 1\right\}\bigcup\left\{(x_1,x_2,x_3)\in S^2\,\Big|\,x_3\ge \frac{1}{2},\;|x_1|<\frac{1}{4}\right\}.
$$
Let $U_4\subset S^2$ be an open topological disk which contains $U_1\cap U_2^c\cap U_3^c$, $U_2\cap U_3^c\cap U_1^c$ and $U_3\cap U_1^c\cap U_2^c$.
Then it is easily checked that the new cover $\mathcal{V}:=\{U_1,U_2,U_3,U_4\}$ of $S^2$ cannot be $3$-localized
and therefore does not satisfy the conditions of \cite[Theorem 1.4.12]{Pay19}.
On the other hand, we can require $U_4$ to have area close to that of union of those three subsets in (\ref{e:converse+})
 so that $\{U_i, U_4\}$ cannot cover $S^2$ for each $i=1,2,3$.
 Then we still have $\kappa(\mathcal{V})=3$, and therefore $\mathcal{V}$ also satisfies the conditions of Theorem~\ref{th:1.4}.


\begin{remark}\label{rm:Ka}
{\rm
As with Buhovsky-Logunov-Tanny \cite{BLT17} and Payette \cite{Pay18,Pay19},
our results are estimation about lower bounds of the Poisson bracket invariant for open covers consisting of topological disks on closed connected surfaces.
On the other hand, Kawasaki \cite[Theorem~1.18]{Ka18} gave a lower bound of the Poisson bracket invariant on closed symplectic manifolds of arbitrary dimension.
As an application, \cite[Example~1.19]{Ka18} showed that the Poisson bracket invariant can be still
positive even if  some elements of an open cover are annuli.
}
\end{remark}

We shall complete proofs of Theorems~\ref{th:1.3},~\ref{th:1.4} in Sections~\ref{sec:2},~\ref{sec:3},
respectively.

\section{Proof of Theorem~\ref{th:1.3}}\label{sec:2}
\setcounter{equation}{0}


Our proof will be completed by the similar argument to \hbox{\cite[Theorem~1.7]{BLT17}}.

\begin{lemma}[\hbox{\cite[Lemma~1.3]{BLT17}}]\label{le:2.1}
Let $(M^{2n},\omega)$ be a closed symplectic manifold of dimension $2n$. Then there exists a constant $c(n)>0$ depending only on the dimension $2n$ of $M$ such that for every finite collection of smooth functions $\{f_i\}_{i=1}^{N}$ on $M$,
$$
\max_{a,b\in[-1,1]^N}\left\|\left\{\sum_{i=1}^{N}a_if_i,\sum_{j=1}^{N}b_jf_j\right\}\right\|\ge c(n)\cdot \max_{M}\sum_{i,j=1}^{N}|\{f_i,f_j\}|.
$$
\end{lemma}


\begin{lemma}[\hbox{\cite[Lemma~2.1]{BLT17}}]\label{le:2.2}
Let $M$ be a closed connected surface and $\omega$ be an area form. For given two smooth functions $f,g:M\rightarrow\mathbb{R}$, let $\Phi:=(f,g):M\rightarrow\mathbb{R}^2$.
Consider  the function $K:\mathbb{R}^2\rightarrow\mathbb{R}\cup\{\infty\}$ defined by
$$
K(s,t):=\#(f^{-1}(s)\cap g^{-1}(t))=\#\Phi^{-1}(s,t).
$$
Then for any Lebesgue measurable set $\Omega\subset\mathbb{R}^2$,
$$
\int_{\Phi^{-1}(\Omega)}|\{f,g\}|\omega=\int_{\Omega}K(s,t)dsdt.
$$
\end{lemma}

 Let $\mathcal{F}=\{f_i\}_{i=1}^{N}$ be a partition of unity
subordinated to a given open cover $\mathcal{U}=\{U_i\}^N_{i=1}$
of $M$.
For $t\ge 0$, we define
\begin{equation}\label{e:2.1}
  U_i(t):=\{x\in M\,|\,f_i>t\},\quad i=1,\cdots, N.
  \end{equation}

 Theorem~\ref{th:1.3} may be derived from Lemma~\ref{le:2.1} and the following:

\begin{lemma}\label{le:2.3}
Under the assumptions of Theorem~{\em \ref{th:1.3}},
 for any partition of unity $\mathcal{F}=\{f_i\}_{i=1}^N$ subordinated to $\mathcal{U}$, it holds that
$$\int_{M}\sum_{i,j=1}^{N}|\{f_i,f_j\}|\omega\ge |\mathcal{J}(\mathcal{U})|, $$
$$\max_M\sum_{i,j=1}^N|\{f_i,f_j\}|\ge\frac{1}{\min_{U_i\in\mathcal{J}(\mathcal{U})}{\rm Area}(U_i)},$$
where  the minimum of an empty set is understand to be infinity.
\end{lemma}

\begin{proof}
This will be proved as in \hbox{\cite[Theorem~1.7]{BLT17}}.
 Since the
 open cover $\mathcal{U}=\{U_i\}_{i=1}^{N}$ contains at least one essential set, we can fix an essential set
 $U_i\in\mathcal{U}$.
 Then there exists a point $z_i\in U_i$ such that $z_i\notin U_j$
 for all $j\in\{1,\cdots,N\}\setminus\{i\}$. It follows that  $f_i(z_i)=1$, and hence $z_i\in U_i(s)$ for all $s\in(0,1)$.
 For every regular value $s\in(0,1)$ of $f_i$,  let $V_i(s)$ denote the connected component of $U_i(s)$ that contains $z_i$.
  Denote by  $\tilde{V}_i(s)$  the open topological disc of minimal area which contains $V_i(s)$ and is contained in $U_i$,
  and by  $\gamma^s:=\partial\tilde{V}_i(s)$.
  Then $\gamma^s$ is connected and is contained in  $\{f_i=s\}\cap \partial V_i(s)$. Fix $y^s\in\gamma^s$ and set
  $t^s_j:=f_j(y^s)\in\mathbb{R}$ for each $j\neq i$.

  For a fixed $j\neq i$, and let $t\in (0, t_j^s)$ be a regular value of $f_j$. Then $y^s\in U_j(t)$.
  Denote by $D_j(t)$ the closure of connected component of $U_j(t)$ that contains $y^s$, and
   by $\tilde{D}_j(t)$ the closed topological disc of minimal area which contains $D_j(t)$
   and is contained in $U_j$. Then $\partial\tilde{D}_j(t)\subset\partial D_j(t)\cap\{f_j=t\}$.

We claim that  $\gamma^s$ has at least two points of intersection with $\partial\tilde{D}_j(t)$.
In fact, since the interior of $\tilde{D}_j(t)$ intersects $\gamma^s$,
we need only to prove that $\gamma^s$ is not contained in $\tilde{D}_j(t)$.
  Equivalently, it suffice to show that both $\tilde{V}_i(s)$ and its complement $\tilde{V}_i(s)^c$ are not contained in $\tilde{D}_j(t)$. Since $D_j(t)\subset U_j$ and $\partial\tilde{D}_j(t)\subset \{f_j=t\}\subset U_j$,
we have $\tilde{D}_j(t)\subset U_j$. Moreover, $\tilde{V}_i(s)$ contains $z_i\notin U_j$, and so
  $\tilde{V}_i(s)\nsubseteq\tilde{D}_j(t)$. In order to get the desired claim, we also need to show that $\tilde{V}_i(s)^c\nsubseteq \tilde{D}_j(t)$.
  Suppose that  $U_i^c\subset U_{j_0}$ for some $j_0\in \{1,\cdots,N\}\setminus\{i\}$.
  Then $\{U_i,U_{j_0}\}$ covers $M$, and thus $\kappa(\mathcal{U})\leqslant 2$, which contradicts the assumption
   $\kappa(\mathcal{U})\ge 3$. Hence, $U_i^c\nsubseteq U_j$ for any $j$ with $j\neq i$. Note that
   $U_i^c\subset \tilde{V}_i(s)^c$ and $\tilde{D}_j(t)\subset U_j$. We conclude that $\tilde{V}_i(s)^c\nsubseteq \tilde{D}_j(t)$.

Denote by ${\rm cv}(f_k)$ the set of critical values of $f_k$, $k=1,\cdots,N$. Set
$$
\Omega_{ij}:=\{(s,t)\,|\,s\in(0,1)\setminus{\rm cv}(f_i)\;\&\; t\in(0,t^s_j)\setminus{\rm cv}(f_j)\}
$$
and $\Phi_{ij}=(f_i,f_j)$ and
$$
K_{ij}:\mathbb{R}^2\to\mathbb{R}\cup\{\infty\},\;(s,t)\mapsto\#\Phi_{ij}^{-1}(s,t)=\#(f_i^{-1}(s)\cap f_j^{-1}(t)).
$$
The above claim  implies that $K_{ij}(s,t)\ge 2 $ for any $(s,t)\in\Omega_{ij}$.
Applying Lemma~\ref{le:2.2} to $f_i$ and $f_j$ with $\Omega:=\Omega_{ij}$, we obtain
$$\begin{aligned}
\int_M|\{f_i,f_j\}|\omega &\ge\int_{\Phi_{ij}^{-1}(\Omega_{ij})}|\{f_i,f_j\}|\omega \\
&=\int_{\Omega_{ij}} K_{ij}(s,t)dsdt\\
&\ge\int_{(0,1)\setminus{\rm cv}(f_i)}ds\int_{(0,t^s_j)\setminus{\rm cv}(f_j)}2dt=2\int_{(0,1)\setminus{\rm cv}(f_i)}t^s_jds.\\
\end{aligned}
$$
Since $t_j^s=f_j(y^s)$, summing the above inequality over all $j\neq i$, we get
\begin{eqnarray*}
\sum_{j=1}^{N}\int_M|\{f_i,f_j\}|\omega &\ge &2\sum_{j\neq i}\int_{(0,1)\setminus{\rm cv}(f_i)}t^s_jds\\
&=&2\int_{(0,1)\setminus{\rm cv}(f_i)}\sum_{j\neq i}f_j(y^s)ds\\
&=&2\int_{(0,1)\setminus{\rm cv}(f_i)} (1-f_i(y^s))ds.
\end{eqnarray*}
For every $s\in (0,1)\setminus{\rm cv}(f_i)$,
since $y^s\in\gamma^s\subset\{f_i=s\}$, we have $f_i(y^s)=s$. Combining this and Morse-Sard theorem (\cite[page 69]{Hi94}), we get
$$
\sum_{j=1}^{N}\int_M|\{f_i,f_j\}|\omega\ge 2\int_{(0,1)\setminus{\rm cv}(f_i)}(1-s) ds=
 2\int^1_0(1-s) ds=1.
 $$
Summing over all $i$ satisfying $U_i\in\mathcal{J}(\mathcal{U})$, we obtain the first inequality in Lemma~\ref{le:2.3}.

In order to prove the second inequality, let us choose  $i_0$ such that $U_{i_0}\in\mathcal{J}(\mathcal{U})$
and ${\rm Area}(U_{i_0})={\min_{U_i\in\mathcal{J}(\mathcal{U})}{\rm Area}(U_i)}$.
Then, since the support of $f_{i_0}$ is contained $U_{i_0}$, we deduce that
\begin{eqnarray*}
1&\le&\sum_{j=1}^{N}\int_M|\{f_{i_0},f_j\}|\omega\\
&=&\sum_{j=1}^{N}\int_{U_{i_0}}|\{f_{i_0},f_j\}|\omega\\
&\le&\max_{U_{i_0}}\sum_{j=1}^{N}|\{f_{i_0},f_j\}|\cdot{\rm Area}(U_{i_0})\\
&\le&\max_M\sum_{i,j=1}^N|\{f_i,f_j\}| \cdot{\rm Area}(U_{i_0}) .
\end{eqnarray*}
\end{proof}

\begin{remark}\label{rm:2.4}
{\rm Actually, there is a question about integrability of the function
$(0,1)\setminus{\rm cv}(f_i)\ni s\mapsto t^s_j\in\mathbb{R}$ in the above proof.
This can be solved by suitable choices of $y^s$.
Recall that we first fix $y^s\in\gamma^s$ and set
  $t^s_j:=f_j(y^s)\in\mathbb{R}$ for each $j\neq i$.
  Since $(0,1)\setminus{\rm cv}(f_i)$ is union of at most countable open intervals contained in $(0,1)$,
  saying $(0,1)\setminus{\rm cv}(f_i)=\cup_{n=1}^\infty I_n$, for each non-empty $I_n$, we
  fix $s_n\in I_n$ and $y^{s_n}\in\gamma_{s_n}$.
  Let $\Phi^{s_n}$ be the gradient flow of $f_i$ through $y^{s_n}$ with respect to some fixed Riemannian metric on $M$.
For each $s\in I_n\setminus\{s_n\}$, let us choose $y^s$ to be the unique intersection point
of $f_i^{-1}(s)$ with ${\rm Im}(\Phi^{s_n})$. Clearly, $y^s$ smoothly depends on $s\in I_n\setminus\{s_n\}$.
Hence $I_n\setminus\{s_n\}\ni s\mapsto t^s_j:=f_j(y^s)\in\mathbb{R}$ is smooth for each $j\neq i$.
}
\end{remark}

\section{Proof of Theorem~\ref{th:1.4}}\label{sec:3}
\setcounter{equation}{0}

We begin with some definitions in \cite{BLT17}.

\begin{definition}[\hbox{\cite[Definition~3.3]{BLT17}}]\label{def:3.1}
{\rm A subset $U$ in a smooth closed connected surface $M$ is said to have a {\it piecewise smooth boundary}
 if $\partial U$ is a finite union of disjoint curves $\Gamma_1,\cdots,\Gamma_m$, such that each $\Gamma_j$ is a simple, closed, piecewise smooth and regular curve.}
\end{definition}

\begin{definition}[\hbox{\cite[Definition~3.3]{BLT17}}]\label{def:3.2}
{\rm Let $M$ be smooth closed connected surface and $\gamma_1,\cdots,\gamma_m \subset M$ be a finite collection of smooth regular curves with a finite number of intersection points. Denote $\Gamma=\gamma_1\cup\cdots\cup\gamma_m$. A connected component of the complement $M\setminus\Gamma$ is called a {\it face} of $\Gamma$. A point $v\in\Gamma$ that lies in the intersection of two (or more) curves is called a {\it vertex} of $\Gamma$.}
\end{definition}

\begin{definition}[cf. \hbox{\cite[Definition~3.1]{BLT17}}]\label{def:3.3}
{\rm A cover  $\mathcal{U}=\{U_i\}_{i\in I}$ of smooth closed connected surface $M$ is called {\it good} if  $\partial U_i$ and $\partial U_j$ intersect transversally for all $i,j\in I$ with $i\neq j$.
Two covers $\mathcal{U}=\{U_i\}_{i\in I},\mathcal{V}=\{V_j\}_{j\in J}$ of $M$ are said to be in {\it generic position} if
$$
\partial U_i \pitchfork \partial U_j\;\hbox{for all }(i,j)\in I\times J,\quad
\partial U_i\cap\partial U_k\cap\partial V_j=\emptyset\quad\hbox{and}\quad\partial U_i\cap\partial V_j\cap\partial V_l=\emptyset
$$
for all $i,k\in I, i\neq k$, and $j,l\in J, j\neq l$.}
\end{definition}

Note that \cite[Definition~3.1]{BLT17} contains no the requirement that $\partial U_i \pitchfork \partial U_j$ for all $(i,j)\in I\times J$.

The following lemma is a variant of \cite[Lemmas~3.2 and 3.4]{BLT17}.

\begin{lemma}\label{le:3.4}
Let $\mathcal{U}=\{U_i\}_{i\in I}$ and $\mathcal{V}=\{V_j\}_{j\in J}$ be good finite open covers of a closed connected surface $M$.
Assume
\begin{description}
\item[(i)] $\mathcal{U}$ and $\mathcal{V}$ are in {\rm generic position}{\rm ;}
\item[(ii)] there exist covers $\widetilde{\mathcal{U}}=\{\widetilde{U}_i\}_{i\in I}$ and $\widetilde{\mathcal{V}}=\{\widetilde{V}_j\}_{j\in J}$ both consisting of topological discs
 such that $U_i\subset\widetilde{U}_i$ and $V_j\subset\widetilde{V}_j$ for all $(i,j)\in I\times J$, and that
 any $\{\widetilde{U}_i, \widetilde{V}_j\}$ cannot cover $M${\rm ;}
\item[(iii)]  there exists $L\in\mathbb{N}$ such that for any point $x\in M$,
$$
\#\{i\in I\,|\,x\in U_i\}\ge L\quad\hbox{and}\quad \#\{j\in J:x\in V_j\}\ge L.
$$
\end{description}
 Then,
$$
\#\bigcup_{i,j}(\partial U_i\cap\partial V_j)\ge 2(L+1)^2.
$$
\end{lemma}

\begin{proof}
Let $I=\{1,2,\cdots, |I|\}$ and $J=\{1,2,\cdots,|J|\}$.
Denote by $S_I$ and $S_J$ the sets of permutations on the elements of $I$ and $J$ respectively.
For $\alpha\in S_I$ and $\beta\in S_J$, define the unions of curves
\begin{eqnarray*}
&&\Gamma_{\alpha}:=\bigcup_{i\in I}(\partial U_{\alpha(i)}\cap U^c_{\alpha(i-1)}\cap\cdots\cap U^c_{\alpha(1)}), \\
&&\Gamma'_{\beta}:=\bigcup_{j\in J}(\partial V_{\beta(j)}\cap V^c_{\beta(j-1)}\cap\cdots\cap V_{\beta(1)}^c).
\end{eqnarray*}

\noindent{\bf Step 1}. {\it  Given  a non-empty connected component $P$ of $M\setminus\Gamma_{\alpha}$
{\rm (}resp. $M\setminus\Gamma'_{\beta}${\rm )}, we prove that there exists  $i\in I$ {\rm (}resp. $j\in J${\rm )}
such that $P\subset U_i$ {\rm (}resp. $P\subset V_j${\rm )}.}
We only prove the case for the component of $M\setminus\Gamma_{\alpha}$.
Another case is similar. Assume that $P\nsubseteq U_i$ for all $i\in I$.
  Since $\partial U_{\alpha(1)}\subset U^c_{\alpha(1)}\subset\Gamma_{\alpha}$,
  we get $P\cap\partial U_{\alpha(1)}\subset P\cap\Gamma_{\alpha}=\emptyset$. This and the fact that
   $P\nsubseteq U_{\alpha(1)}$ imply that $P\subset U^c_{\alpha(1)}$. Assuming $P\subset U^c_{\alpha(1)}\cap\cdots\cap U^c_{\alpha(i-1)}$,
    we conclude that $P\cap\partial U_{\alpha(i)}\subset P\cap\Gamma_{\alpha}=\emptyset$,
    since $\partial U_{\alpha(i)}\cap  U^c_{\alpha(1)}\cap\cdots\cap U^c_{\alpha(i-1)}\subset\Gamma_{\alpha}$.
      As above, it follows  that $P\subset U^c_{\alpha(i)}$, since $P\nsubseteq U_{\alpha(i)}$.
   By induction on $i\in I$, we obtain that $P\subset\cap_{i\in I}U^c_{\alpha(i)}=\emptyset$,
     which contradicts $P\ne\emptyset$.

\noindent{\bf Step 2}. {\it Proof of $\#(\Gamma_{\alpha}\cap\Gamma'_{\beta})\ge 1$}. First, let us show that by removing out parts from $\Gamma_{\alpha}$ and $\Gamma'_{\beta}$, we can get $\widetilde{\Gamma}_{\alpha}$ and $\widetilde{\Gamma}'_{\beta}$ such that their faces are open topological discs.
By Step 1, for any face $P$ of $\Gamma_{\alpha}$, there exists $i\in I$ such that $P\subset U_i\subset\widetilde{U}_i$.
Let $\widetilde{P}$ be the open topological disc of minimal area that contains $P$ and is contained in $\widetilde{U}_i$. Then $\partial{\widetilde{P}}\subset\partial P\subset\Gamma_{\alpha}$, and   $\widetilde{P}$ is a face of $\Gamma_\alpha^1:=\Gamma_\alpha\setminus(\Gamma_{\alpha}\cap\widetilde{P})$.
For a face  ${P}_1$ of $\Gamma_\alpha^1$ which is not an open topological disc,
 repeating the above argument, we finally  get $\widetilde{\Gamma}_{\alpha}$
 such that each face of $\widetilde{\Gamma}_{\alpha}$ is  an open topological disc.
 Similarly, starting from $\Gamma'_{\beta}$, we obtain a $\widetilde{\Gamma}'_{\beta}$ with the same properties.
 It is obvious that
 $$
 \#(\Gamma_{\alpha}\cap\Gamma'_{\beta})\ge\#(\widetilde{\Gamma}_{\alpha}\cap\widetilde{\Gamma}'_{\beta}).
 $$

Call a face $G$ of $\widetilde{\Gamma}_{\alpha}$ (resp. $\widetilde{\Gamma}'_{\beta}$) is {\it maximal} if it is not properly contained in any face of
$\widetilde{\Gamma}'_{\beta}$ (resp. $\widetilde{\Gamma}_{\alpha}$).
 Notice that any non-maximal face of $\widetilde{\Gamma}_{\alpha}$ is contained in a maximal face of $\widetilde{\Gamma}'_{\beta}$.
 (Indeed,  for a non-maximal face $P$ of $\widetilde{\Gamma}_{\alpha}$,
 there exists a face $Q$ of $\widetilde{\Gamma}'_{\beta}$ such that $P\subsetneq Q$ by the definition. We claim that $Q$ is a maximal face of $\widetilde{\Gamma}'_{\beta}$.
  Otherwise, there exists a face $P'$ of $\widetilde{\Gamma}_{\alpha}$ such that $P\subsetneq Q\subsetneq P'$,
  which contradicts the fact that $P$ is a face and thus a connected component.)  Therefore, the union of maximal faces of both $\widetilde{\Gamma}_{\alpha}$ and $\widetilde{\Gamma}'_{\beta}$ covers the complementary of $\widetilde{\Gamma}_{\alpha}\cap\widetilde{\Gamma}'_{\beta}$ in
  $M$. Hence, either $\widetilde{\Gamma}_{\alpha}$ or $\widetilde{\Gamma}'_{\beta}$ has at least one maximal face.
     Without loss of generality, we assume that  $\widetilde{\Gamma}_{\alpha}$
    has maximal faces.

Let $G$ be a maximal face of $\widetilde{\Gamma}_{\alpha}$. Then
 there exists a face $G'$ of $\widetilde{\Gamma}'_{\beta}$ such that
 \begin{equation}\label{e:3.1}
 \partial G\cap G'\neq\emptyset.
 \end{equation}
 (Otherwise, $\partial G\subset\widetilde{\Gamma}'_{\beta}$, and thus $\#(\widetilde{\Gamma}_{\alpha}\cap\widetilde{\Gamma}'_{\beta})=\infty$,
 the conclusion is proved.) We claim that
 \begin{equation}\label{e:3.2}
 \partial G\cap(M\setminus G')\neq\emptyset.
  \end{equation}
  In fact, if $\partial G\cap(M\setminus G')=\emptyset$, then $\partial G\subset G'$.
   We have either $M\setminus G'\subset G$ or $G\subset G'$, since $M\setminus G'$ is connected. The first case is impossible by the assumption (ii). In the second case, we get $G\subset G'$,
    which contradicts the maximality of $G$.

 Since the boundaries $\partial G$ and $\partial G'$ are simple closed curves
 and also intersect transversely, (\ref{e:3.1}) and (\ref{e:3.2}) imply $\#(\partial G\cap\partial G')\ge 2$. Therefore,
 $$
 \#(\Gamma_{\alpha}\cap\Gamma'_{\beta})\ge\#(\widetilde{\Gamma}_{\alpha}\cap\widetilde{\Gamma}'_{\beta})\ge\#(\partial G\cap\partial G')\ge 2.
 $$

\noindent{\bf Step 3}. Since $\Gamma_{\alpha}\subset\cup_{i\in I}\partial U_i$ and $\Gamma'_{\beta}\subset\cup_{j\in J}\partial V_j$, we have $\Gamma_{\alpha}\cap\Gamma'_{\beta}\subset\cup_{i,j}(\partial U_i\cap\partial V_j)$ and
\begin{equation}\label{e:3.3}
\triangle:=\bigcup_{\alpha\in S_I}\bigcup_{\beta\in S_J}\Gamma_{\alpha}\cap\Gamma'_{\beta}\subset\cup_{i,j}(\partial U_i\cap\partial V_j).
  \end{equation}
Fixing a point $x\in\triangle$, we estimate
$$
\#\{(\alpha,\beta)\in S_I\times S_J\,|\, x\in\Gamma_{\alpha}\cap\Gamma'_{\beta}\}=\#\{\alpha\in S_I\,|\, x\in\Gamma_\alpha\}\times
\#\{\beta\in S_J\,|\, x\in\Gamma'_\beta\}.
$$
 Let $i\in I$ be such that $x\in\partial U_i$. By the definition of $\Gamma_\alpha$,
  $x\in\Gamma_{\alpha}$ only if $\alpha^{-1}(i)<\alpha^{-1}(k)$ for any $k\in I$ such that $x\in U_k$. By our assumption, $\#\{k\in I\,|\, x\in U_k\}\ge L$. Because of symmetry, the number of permutations $\sigma=\alpha^{-1}$ for which $\#\{k\in I\,|\,\sigma(i)<\sigma(k)\}\ge L$  is at most $\frac{|I|!}{(L+1)}$.
  It follows that
  $$
  \#\{\alpha\in S_I\,|\, x\in\Gamma_\alpha\}\le \frac{|I|!}{(L+1)}.
  $$
    Similarly, we have $\#\{\beta\in S_J\,|\, x\in\Gamma_\beta\}\le \frac{|J|!}{(L+1)}$.

   Note that the disjoint union
   $$
   \coprod_{\alpha\in S_I}\coprod_{\beta\in J}(\Gamma_{\alpha}\cap\Gamma'_{\beta})
   $$
   has exactly
   $$
   \sum_{\alpha\in S_I}\sum_{\beta\in S_J}\#(\Gamma_{\alpha}\cap\Gamma'_{\beta})
   $$
    elements, and that
   each point $x\in\triangle$
   appears $\#\{(\alpha,\beta)\in S_I\times S_J\,|\, x\in\Gamma_{\alpha}\cap\Gamma'_{\beta}\}$ times
   in $\coprod_{\alpha\in S_I}\coprod_{\beta\in S_J}(\Gamma_{\alpha}\cap\Gamma'_{\beta})$. Then
   \begin{eqnarray*}
   &&\#\left(\bigcup_{\alpha\in S_I}\bigcup_{\beta\in S_J}\Gamma_{\alpha}\cap\Gamma'_{\beta}\right)\times
   \frac{|I|!}{(L+1)}\times\frac{|J|!}{(L+1)}\\
   &\ge&\#\left(\bigcup_{\alpha\in S_I}\bigcup_{\beta\in S_J}\Gamma_{\alpha}\cap\Gamma'_{\beta}\right)\times
   \max\left\{\#\{(\alpha,\beta)\in S_I\times S_J\,|\, x\in\Gamma_{\alpha}\cap\Gamma'_{\beta}\}\,|\,
   x\in\triangle\right\}\\
   &\ge& \#\left(\coprod_{\alpha\in S_I}\coprod_{\beta\in S_J}(\Gamma_{\alpha}\cap\Gamma'_{\beta})\right)\\
   &=&\sum_{\alpha\in S_I}\sum_{\beta\in S_J}\#(\Gamma_{\alpha}\cap\Gamma'_{\beta})
   \end{eqnarray*}
    and hence
  \begin{eqnarray*}
\#\bigcup_{i,j}\left(\partial U_i\cap\partial V_j\right)&\ge&
\#\left(\bigcup_{\alpha\in S_I}\bigcup_{\beta\in S_J}\Gamma_{\alpha}\cap\Gamma'_{\beta}\right)\\
&\ge&\frac{L+1}{|I|!}\frac{L+1}{|J|!}\sum_{\alpha\in S_I}\sum_{\beta\in S_J}\#(\Gamma_{\alpha}\cap\Gamma'_{\beta})\\
&\ge& 2(L+1)^2.
\end{eqnarray*}
\end{proof}

In order to prove Theorem~\ref{th:1.4}, we need the following corresponding result of \cite[Theorem~$1.5^\prime$]{BLT17}.

\begin{theorem}\label{th:1.4+}
Let $(M,\omega)$ be a closed connected symplectic surface and let
$\mathcal{U}=\{U_i\}_{i\in I}$ and $\mathcal{V}=\{V_j\}_{j\in J}$
be two finite open covers  consisting of topological discs. Suppose that $\{U_i, V_j\}$ is not an open cover of $M$
for each pair $(i,j)\in I\times J$. Then for any partition of unity $\mathcal{F}:=\{f_i\}_{i\in I}$ subordinate to $\mathcal{U}$,
and $\mathcal{G}:=\{g_j\}_{j\in J}$ subordinate to $\mathcal{V}$, there holds
$$
\int_M\sum_{i\in I}\sum_{j\in J}|\{f_i,g_j\}|\omega\ge 2.
$$
\end{theorem}

\begin{remark}\label{rm:3.6}
{\rm In \cite{Pay19+}, J. Payette told us that the lower bound in
 Lemma 3.4 of the original version  could be multiplied by $2$.  As a
result, we can prove Theorem~\ref{th:1.4+} and thus improve
the lower bounds in the original Theorem 1.7.}
\end{remark}

We postpone the proof of this theorem until the end of this section.

\begin{proof}[Proof of Theorem~\ref{th:1.4}]
Since $\kappa(\mathcal{U})\ge 3$, any two open subsets in $\mathcal{U}$ cannot cover $M$.
Hence, applying Theorem~\ref{th:1.4+} to $\mathcal{U}$ and $\mathcal{V}=\mathcal{U}$, we may obtain
$$
\int_M\sum^N_{i,j=1}|\{f_i,f_j\}|\omega\ge 2
$$
for any partition of unity $\mathcal{F}:=\{f_i\}^N_{i=1}$ subordinate to $\mathcal{U}$.

As to the second claim, by Lemma~\ref{le:2.1} for $n=1$, we deduce that
  \begin{eqnarray*}
  \nu_c(\mathcal{F})&=&\max_{a,b\in[-1,1]^N}\left\|\left\{\sum_{i=1}^{N}a_if_i,\sum_{j=1}^{N}b_jf_j\right\}\right\|,\\
   &\ge&c(1) \cdot \max_{M}\sum_{i,j=1}^{N}|\{f_i,f_j\}|\\
   &\ge&\frac{c(1)}{{\rm Area}(M)} \int_M\sum^N_{i,j=1}|\{f_i,f_j\}|\omega \\
   &=&\frac{2c(1)}{{\rm Area}(M)}.
   \end{eqnarray*}
Set $C:=2c(1)$. Because of the arbitrariness of $\mathcal{F}$, this leads to
$$
{\rm pb}(\mathcal{U})\ge \frac{C}{{\rm Area}(M)}.
$$
\end{proof}

\begin{proof}[Proof of Theorem~\ref{th:1.4+}]
The basic ideas are the same as those of \hbox{\cite[Theorem~${1.5}^\prime$]{BLT17}}.
Fix a large $L\in\mathbb{N}$ such that $L> |I|+|J|$.  We shall use the functions $\{f_i\}_{i\in I}$ and
$\{g_j\}_{j\in J}$ to construct covers satisfying the assumptions of Lemma~\ref{le:3.4}. For every $(i,j)\in I\times J$, pick $m_i, n_j\in\mathbb{N}$ such that
$$
\frac{m_i}{L}>\max_{M}f_i\quad\hbox{and}\quad \frac{n_j}{L}>\max_{M}g_j.
$$
Consider the intervals
\begin{eqnarray*}
&&\mathcal{I}_{i,k}:=\left[\frac{k-1}{L},\frac{k}{L}\right],\quad
k=1,\cdots, m_i,\\
&&\mathcal{J}_{j,l}:=\left[\frac{l-1}{L},\frac{l}{L}\right],\quad
l=1,\cdots, n_j.
\end{eqnarray*}
  Denote by $s_{i,k}\in\mathcal{I}_{i,k}$ (resp. $t_{j, l}\in\mathcal{J}_{j,l}$) an independent variable. 
  We equip the intervals $\mathcal{I}_{i,k}$ and $\mathcal{J}_{j,l}$
  with the normalized Lebesgue measure $\mu_{i,k}:=Lds_{i,k}$ and $\nu_{j,l}:=Ldt_{j,l}$, respectively.
 Put $m:=\sum_{i\in I}m_i$ and $n:=\sum_{j\in J}n_j$. Then
$$
\mathcal{C}:=\prod_{i\in I}\prod_{1\leqslant k\leqslant m_i}\mathcal{I}_{i,k}\subset\mathbb{R}^m\quad\hbox{and}\quad
\mathcal{D}:=\prod_{j\in J}\prod_{1\leqslant l\leqslant n_j}\mathcal{J}_{j,l}\subset\mathbb{R}^n.
$$
 For $s:=(s_{i,k})_{i,k}\in\mathcal{C}$ and $t:=(t_{j,l})_{j,l}\in\mathcal{D}$, consider the open sets
 \begin{eqnarray*}
&&U_{i,k}^s:=U_{i,k}(s_{i,k})=\{f_i>s_{i,k}\},\quad 1\leqslant k\leqslant m_i,\; i\in I, \\
&&V_{j,l}^t:=V_{j,l}(t_{j,l})=\{g_j>t_{j,l}\},\quad 1\leqslant l\leqslant n_j,\; j\in J.
\end{eqnarray*}
Since $\sum_{i\in I}f_i=1$ and $\sum_{j\in J}g_j=1$, for any $x\in M$, there exist $i\in I$ and $j\in J$ such that
$$
f_i(x)\ge \frac{1}{|I|}>\frac{1}{L}>s_{i,1}\quad\hbox{and}\quad g_j(x)\ge \frac{1}{|J|}>\frac{1}{L}>t_{j,1}.
$$
 It follows that both $\mathcal{U}^s:=\{U_{i,k}^s\}_{i,k}$ and $\mathcal{V}^t:=\{V_{j,l}^t\}_{j,l}$ are open covers of $M$ for any
 $(s,t)\in\mathcal{C}\times\mathcal{D}$.

 Let us show that after $L$ is replaced by $\hat{L}:=L-|I|-|J|$, the covers $\mathcal{U}^s$ and $\mathcal{V}^t$
  satisfy the assumptions of Lemma~\ref{le:3.4} for almost all  $(s,t)\in\mathcal{C}\times\mathcal{D}$.
   The condition (iii) is easily checked. Given  $x\in M$,  for every $i\in I$, we have
$$\begin{aligned}
\#\left\{1\leqslant k\leqslant m_i\,|\,x\in U_{i,k}^s\right\}&= \#\left\{1\leqslant k\leqslant m_i\,|\, f_i(x)> s_{i,k}\right\}\\
&\ge\#\left\{1\leqslant k\leqslant m_i\,\Big|\, f_i(x)>\frac{k}{L}\right\} \\
&\ge Lf_i(x)-1.
\end{aligned}
$$
This implies that the number of sets in $\mathcal{U}^s$ containing $x$ is at least
$$
\sum_{i\in I}(Lf_i(x)-1)=L-|I|.
$$
Similarly,  the number of sets in $\mathcal{V}^s$ containing $x$ is at least $L-|J|$.

\noindent{\bf Claim}. {\it For almost all $(s,t)\in\mathcal{C}\times\mathcal{D}$, the cover $\mathcal{U}^s$ and $\mathcal{V}^t$ are good and in generic positions.}

 Indeed, consider maps
  \begin{eqnarray*}
&& \Phi_{i,i'}:M\longrightarrow\mathbb{R}^2,\;x\mapsto(f_i(x), f_{i'}(x)),\qquad\hbox{for all } i,i',\\
 && \Psi_{j,j'}:M\longrightarrow\mathbb{R}^2,\;x\mapsto(g_j(x), g_{j'}(x)),\qquad\hbox{for all } j,j',\\
 && \Upsilon_{i,j}:M\longrightarrow\mathbb{R}^2,\;x\mapsto(f_i(x), g_j(x)),\qquad\hbox{for all } i,j.
 \end{eqnarray*}
 By Morse-Sard theorem (\cite[page 69]{Hi94}), for almost all $(s,t)\in\mathcal{C}\times\mathcal{D}$,
we arrive at:  $(s_{i,k}, s_{i',k'})$ is a common regular value of all maps $\Phi_{i,i'}$,
  $(t_{j,l},t_{j',l'})$ is that of all maps $\Psi_{j,j'}$,
  and  $(s_{i,k},t_{j,l})$ is that of all maps $\Upsilon_{i,j}$.

  In particular, for such a pair $(s,t)$, the boundaries $\partial U_{i,k}^s$ and $\partial V_{j,l}^t$
  (resp. $\partial U_{i,k}^s$ and $\partial U_{i',k'}^s$,  $\partial V_{j,l}^t$ and $\partial V_{j',l'}^t$)
   intersect transversely.
    Moreover, if $\partial U_{i,k}^s\cap\partial V_{j,l}^t\cap\partial V_{j',l'}^t\ne\emptyset$ (resp.
    $\partial U_{i,k}^s\cap\partial U_{i',k'}^s\cap\partial V_{j,l}^t\ne\emptyset$), then
    $$
    \hbox{$(s_{i,k},t_{j,l},t_{j',l'})$ (resp. $(s_{i,k},s_{i',k'},t_{j,l})$)}
    $$
         is in the image of
    \begin{eqnarray*}
   && \Phi_{i,j,j'}:M\longrightarrow \mathbb{R}^3, x\mapsto (f_i(x),g_j(x),g_{j'}(x))\\
   &&\hbox{(resp. $\Psi_{i,i',j}:M\longrightarrow \mathbb{R}^3, x\mapsto (f_i(x),f_{i'}(x),g_{j}(x))$)}.
    \end{eqnarray*}
     Since the images ${\rm Im}(\Phi_{i,j,j'})$ and ${\rm Im}(\Psi_{i,i',j})$
     have the codimension at least $1$ in $\mathbb{R}^3$
    and thus both  $\cup_{i,j,j'}{\rm Im}(\Phi_{i,j,j'})$ and $\cup_{i,i',j}{\rm Im}(\Psi_{i,i',j})$
    have measure  zero, we deduce that for almost all $(s,t)\in\mathcal{C}\times\mathcal{D}$, the covers $\mathcal{U}^s$ and $\mathcal{V}^t$ are good and in generic positions.

  It remains to prove that (ii) of Lemma~\ref{le:3.4} is satisfied.
    Note that
    $$
    U_{i,k}^s:=\{f_i>s_{i,k}\}\subset {\rm supp}f_i\subset U_i\quad\hbox{and}\quad
    V_{j,l}^t:=\{g_j>t_{j,l}\}\subset {\rm supp}g_j\subset V_j.
    $$
    Taking $\widetilde{U_{i,k}^s}=U_i$ and $\widetilde{V_{j,l}^t}=V_j$, it is easy to see that they satisfy (ii) of Lemma~\ref{le:3.4}.

    In summary,  after $L$ is replaced by $\hat{L}$, $\mathcal{U}^s$ and $\mathcal{V}^t$ satisfy the assumptions of Lemma~\ref{le:3.4} for almost all $(s,t)\in\mathcal{C}\times\mathcal{D}$.
 Therefore, for almost all $(s,t)\in\mathcal{C}\times\mathcal{D}$,  we obtain
$$
\#\cup_{i,k,j,l}\left(\partial U_{i,k}^s\cap\partial V_{j,l}^t\right)\ge 2(\hat{L}+1)^2.
$$
Averaging this inequality over $(s,t)\in\mathcal{C}\times\mathcal{D}$ with respect to the normalized product measure $\mu\times\nu$,
where $\mu:=\prod_{i,k}\mu_{i,k}$ and $\nu:=\prod_{j,l}\nu_{j,l}$, we obtain
$$\begin{aligned}
2(\hat{L}+1)^2&\le\int_{\mathcal{C}\times\mathcal{D}} \#\cup_{i,k,j,l}\left(\partial U_{i,k}^s\cap\partial V_{j,l}^t\right)\,d\mu(s)d\nu(t)\\
&\le\int_{\mathcal{C}\times\mathcal{D}}\sum_{i,k,j,l}\#\left(\partial U_{i,k}^s\cap\partial V_{j,l}^t\right)d\mu(s)d\nu(t) \\
&=\sum_{i,k,j,l}\int_{\frac{k-1}{L}}^{\frac{k}{L}}\int_{\frac{l-1}{L}}^{\frac{l}{L}}\#\left(\partial U_{i,k}^s\cap\partial V_{j,l}^t\right)d\mu_{i,k}(s_{i,k})d\nu_{j,l}(t_{j,l})\\
&=L^2\sum_{i,k,j,l}\int_{\frac{k-1}{L}}^{\frac{k}{L}}\int_{\frac{l-1}{L}}^{\frac{l}{L}}\#\left(\partial U_{i,k}^s\cap\partial V_{j,l}^t\right)ds_{i,k}dt_{j,l},
\end{aligned}
$$
since $\mu_{i,k}:=Lds_{i,k}$ and $\nu_{j,l}:=Ldt_{j,l}$.
Moreover, for any values of $s_{i,k}$ and $t_{j,l}$, we have
$$
\partial U_{i,k}^s=\partial\{f_i>s_{i,k}\}\subset f_i^{-1}(s_{i,k})\quad\hbox{and}\quad
\partial V_{j,l}^t=\partial\{g_j> t_{j,l}\}\subset g_j^{-1}(t_{j,l}).
$$
 Hence,
$$
2(\hat{L}+1)^2\le L^2\sum_{i,k,j,l}\int_{\frac{k-1}{L}}^{\frac{k}{L}}\int_{\frac{l-1}{L}}^{\frac{l}{L}}\#\left(f_i^{-1}(s_{i,k})\cap g_j^{-1}(t_{j,l})\right)ds_{i,k}dt_{j,l}.
 $$
Recall that $\Upsilon_{i,j}=(f_i,g_j):M\rightarrow\mathbb{R}^2$. Set
$$
\Omega_{k,l}:=\left(\frac{k-1}{L},\frac{k}{L}\right)\times\left(\frac{l-1}{L},\frac{l}{L}\right)\subset\mathbb{R}^2.
$$
From Lemma~\ref{le:2.2}, we deduce
$$\begin{aligned}
2(\hat{L}+1)^2&\le L^2\sum_{i,k,j,l}\int_{\Omega_{k,l}}\int_{\frac{l-1}{L}}^{\frac{l}{L}}\#\left(f_i^{-1}(s_{i,k})\cap g_j^{-1}(t_{j,l})\right)ds_{i,k}dt_{j,l}\\
&=L^2\sum_{i,k,j,l}\int_{\Phi^{-1}_{i,j}(\Omega_{k,l})}|\{f_i,g_j\}|\omega\\
&\le L^2\sum_{i,j}\int_M|\{f_i,g_j\}|\omega
\end{aligned}
$$
and thus
$$
\sum_{i,j}\int_M |\{f_i,g_j\}|\omega\ge\frac{2(\hat{L}+1)^2}{L^2}=\frac{2(L+1-|I|-|J|)^2}{L^2}.
$$
Letting $L\rightarrow\infty$,  the conclusion  is proved.
\end{proof}

{\bf Acknowledgements}.\quad
Both authors are deeply grateful to Professor Lev Buhovsky for wonderful reports and explanations on their work. We also thank Dr. Jordan Payette for giving us some valuable advice of modification, which improves the original version. Finally, we would like to thank the anonymous referees for pointing out some mistakes, explicit improving suggestions and other valuable comments.


%

\medskip
\begin{tabular}{l}
Kun Shi\\
 School of Mathematical Sciences, Beijing Normal University\\
 Laboratory of Mathematics and Complex Systems, Ministry of Education\\
 Beijing 100875, The People's Republic of China\\
 E-mail address: shikun@mail.bnu.edu.cn\vspace{5mm}\\
%
Guangcun Lu (Corresponding author)\\ 
 School of Mathematical Sciences, Beijing Normal University\\
 Laboratory of Mathematics and Complex Systems, Ministry of Education\\
 Beijing 100875, The People's Republic of China\\
 E-mail address: gclu@bnu.edu.cn\\
\end{tabular}

\end{document}